\documentclass{amsart}

\usepackage{amsmath,amssymb,amsthm}

\usepackage{pinlabel}

\newtheorem{prop}{Proposition}[section]
\newtheorem{thm}[prop]{Theorem}
\newtheorem{lem}[prop]{Lemma}
\newtheorem{cor}[prop]{Corollary}

\theoremstyle{definition}

\newtheorem{defn}[prop]{Definition}

\newtheorem{rem}[prop]{Remark}
\newtheorem{Note}[prop]{Note}

\newtheorem*{ack}{Acknowledgements}

\def\co{\colon\thinspace}

\newcommand{\talpha}{\tilde{\alpha}}

\newcommand{\calC}{\mathcal{C}}

\newcommand{\rmd}{\mathrm{d}}

\newcommand{\N}{\mathbb{N}}

\newcommand{\PD}{\mathrm{PD}}

\newcommand{\R}{\mathbb{R}}
\newcommand{\RP}{\mathbb{R}\mathrm{P}}
\newcommand{\wtR}{\widetilde{R}}
\newcommand{\rhod}{\rho_{\mathrm{dual}}}

\newcommand{\SL}{\mathrm{SL}}

\newcommand{\Z}{\mathbb{Z}}

\DeclareMathOperator{\Crit}{Crit}

\begin{document}

\author[H.~Geiges]{Hansj\"org Geiges}
\address{Mathematisches Institut, Universit\"at zu K\"oln,
Weyertal 86--90, 50931 K\"oln, Germany}
\email{geiges@math.uni-koeln.de}
\author[J.~Hedicke]{Jakob Hedicke}
\address{Centre de Recherches Math\'ematiques, Universit\'e de Montr\'eal,
Pavillon Andr\'e-Aisenstadt, 2920 Chemin de la tour, Montr\'eal H3T 1J4,
Canada}
\email{jakob.hedicke@gmail.com}
\author[M.~Sa\u{g}lam]{Murat Sa\u{g}lam}
\address{Mathematisches Institut, Universit\"at zu K\"oln,
Weyertal 86--90, 50931 K\"oln, Germany}
\email{msaglam@math.uni-koeln.de}

\title[Bott-integrable overtwisted contact structures]{Bott-integrability
of overtwisted contact structures}

\date{}

\begin{abstract}
We show that an overtwisted contact structure on a closed, oriented
$3$-manifold can be defined by a contact form having a
Bott-integrable Reeb flow if and only if the Poincar\'e dual of its Euler
class is represented by a graph link.
\end{abstract}



\maketitle


\section{Introduction}
In this paper we continue our systematic study of Bott-integrable
Reeb flows on $3$-manifolds begun in~\cite{ghs24}. Slightly
streamlining the terminology used in our earlier paper, we 
speak of Bott integrability of contact forms and contact
structures in the following sense. Our $3$-manifolds are always understood
to be oriented, and our contact structures are coorientable and positive,
that is, $\xi=\ker\alpha$ with $\alpha\wedge\rmd\alpha>0$.

\begin{defn}
A contact form $\alpha$ on a compact $3$-manifold $M$ is \emph{Bott integrable}
if there is a Morse--Bott function $f\co M\rightarrow\R$ invariant
under the Reeb flow of~$\alpha$, that is, $\rmd f(R)=0$ for the
Reeb vector field $R$ of~$\alpha$.

A contact structure $\xi$ on $M$ is \emph{Bott integrable}
if there is some Bott-integrable contact form $\alpha$ with $\xi=\ker\alpha$.
\end{defn}

Our main result in the present paper gives a complete classification
of all Bott-integrable contact structures amongst the overtwisted
ones on any closed $3$-mani\-fold. To formulate it, we need
to explain the concept of a graph link.

Recall that a \emph{graph manifold} is a $3$-manifold that can be cut along
tori into Seifert fibred pieces. In \cite[Theorem~1.3]{ghs24}
we showed that a closed $3$-manifold admits a Bott-integrable contact structure
if and only if it is a graph manifold. We also gave examples
of non-integrable contact structures on graph manifolds, and examples
of $3$-manifolds (the $3$-sphere $S^3$, the $3$-torus $T^3$, and
$S^1\times S^2$) where every contact structure is Bott integrable.

A link in a $3$-manifold is called a \emph{graph link} if its exterior is
a graph manifold. In particular, the ambient manifold containing a graph
link must be a graph manifold, since it is obtained by gluing solid
tori to the link exterior.

\begin{thm}
\label{thm:main}
An overtwisted contact structure on a closed $3$-manifold is Bott integrable
if and only if the Poincar\'e dual of its Euler class can be represented
by a graph link.
\end{thm}

\begin{rem}
Here is a brief comment on the dynamical significance of this
theorem. By a result of Paternain~\cite{pate91}, a Hamiltonian
flow (on a $3$-dimensional energy hypersurface) with an integral 
whose critical set consists of submanifolds has vanishing
topological entropy. In particular, this applies to Bott-integrable
Reeb flows. Macarini--Schlenk~\cite{masc11}, by contrast, have
shown the existence of contact structures that force entropy, 
in the sense that any associated Reeb flow has positive topological
entropy. Other results of this type are due to Alves~\cite{alve16}
and Alves--Colin--Honda~\cite{ach19}. In all cases, the
relevant property of the contact structure that forces entropy also
implies tightness of the contact structure.

In \cite[Question~1.5]{ach19}, the authors ask whether there
exists an \emph{overtwisted} contact structure in dimension~$3$
that forces entropy. As shown by Yau~\cite{yau06},
the contact homology of overtwisted contact structures vanishes.
This suggests that the orbit structure of overtwisted Reeb flows
may not be rich enough to establish forced entropy by the methods
of~\cite{alve16}.

In combination with the result of Paternain,
our Theorem~\ref{thm:main} gives a criterion for an overtwisted
contact structure \emph{not} to force entropy. As an example,
our results in Section~\ref{section:Seifert} show that for
Seifert manifolds the answer to \cite[Question~1.5]{ach19}
is negative: no overtwisted contact structure on a Seifert manifold
forces entropy.

On the other hand,
Theorem~\ref{thm:main} provides examples of overtwisted contact
structures on graph manifolds that force any associated Reeb flow
at least to be non-integrable (in the Bott sense);
see Section~\ref{section:cat}. Now, one of two things can
happen with these examples: Either (i) entropy is forced, providing
a positive answer to \cite[Question~1.5]{ach19}, or (ii) there is a Reeb flow
of zero topological entropy that cannot be approximated by
Bott-integrable Reeb flows. Case (ii) would go some way towards
answering a question of Katok~\cite[Problem~1]{kato09}, who wondered about
the existence of such approximations by Liouville-integrable systems.
\end{rem}

In Sections \ref{section:crit} and \ref{section:Euler} we prove
a few results that will be of more general interest for the study
of Bott-integrable Reeb flows: neighbourhood theorems
for surfaces contained in the critical set of the Bott integral,
perturbation results for such critical surfaces, and a formula
for the Euler class of a Bott-integrable contact structure
in terms of the critical periodic Reeb orbits.

This leads to a proof of the `only if' part of Theorem~\ref{thm:main},
which is in fact a general statement
about Bott-integrable contact structures, not only the overtwisted ones.
Given a Bott-integrable contact form with Bott integral~$f$, write
$L_f$ for the link of critical Reeb orbits. We then show
that $L_f$ is a graph link, and with the components oriented
depending on their elliptic or hyperbolic type, $L_f$ represents
the dual of the Euler class. The fact that this remains true even in
the presence of critical surfaces rests in an essential way on the
neighbourhood and perturbation theorems.

Given a compact graph manifold~$M$, there is a characterisation due
to Yano~\cite{yano85a,yano85b} of the homology classes in $H_1(M)$
that can be represented by a graph link. This is instrumental
for the proof of the `if' part of Theorem~\ref{thm:main}.
In Section~\ref{section:proof} we provide the necessary
background on $3$-manifold topology, notably the JSJ decomposition,
for understanding Yano's result.

In order to apply Yano's result in our setting, in
Section~\ref{section:Seifert} we show
how to realise any $1$-dimensional homology class on a given
compact Seifert fibred manifold by a sublink of the critical link of a
suitable Bott-integrable contact form. This also provides
an alternative contact geometric proof of an auxiliary result
in~\cite{yano85b}. For Seifert manifolds over a non-orientable base,
the proof of this $H_1$-completeness, as we call it, again
uses the perturbation result for critical surfaces. Furthermore, the
argument requires the fibre connected sum of Bott-integrable
contact forms, which is another general construction we introduce here.
A consequence of this discussion is a direct proof (not involving
Yano's results or any advanced $3$-manifold topology used to
establish Theorem~\ref{thm:main}) that every overtwisted contact structure
on any closed Seifert manifold is Bott integrable
(Corollary~\ref{cor:Seifert-ot}).

By Eliashberg's classification~\cite{elia89}, every homotopy class
of tangent $2$-plane fields on a closed, oriented $3$-manifold
contains a unique overtwisted contact structure up to isotopy.
Thus, together with the results of Yano, Theorem~\ref{thm:main}
allows one to give a complete homotopy theoretic classification, on a
given $3$-manifold, of the Bott-integrable overtwisted contact structures.
As an instructive example, in Section~\ref{section:cat} we classify the
Bott-integrable contact structures on the mapping torus of
Arnold's cat map, including the tight ones. This serves to illustrate
Yano's results and the $3$-manifold topology behind them,
and it extends our classification results from~\cite{ghs24}.
In particular, it gives further examples of contact structures (on
a graph manifold) that do \emph{not} admit Bott-integrable Reeb flows.
\section{The critical set}
\label{section:crit}
We consider a pair $(M,\alpha)$ consisting of a closed,
oriented $3$-manifold $M$ and a positive contact form~$\alpha$.
The Reeb vector field of $\alpha$ will be denoted by~$R$.

In \cite[Theorem~2.2]{ghs24} we proved the following
Liouville type theorem:
If $f$ is an integral for the Reeb flow, that is,
$\rmd f(R)=0$, then any regular level set of $f$ is
a $2$-torus~$T^2$, and in terms of circle coordinates $x_1,x_2$
on $T^2$ and a suitable transverse coordinate~$r$, the contact form
$\alpha$ is of Lutz type $\alpha=h_1(r)\,\rmd x_1+
h_2(r)\,\rmd x_2$, and the integral $f$ a function of~$r$. 

For the case when the integral $f$ is a Morse--Bott function,
we also established a neighbourhood theorem for isolated elliptic
Reeb orbits in the critical set $\Crit(f)$, that is,
an orbit where the Bott integral has a minimum or a maximum;
see \cite[Theorem~2.4]{ghs24}: In terms of suitable
coordinates $(\theta;r,\varphi)$ on a neighbourhood
$S^1\times D^2_{\delta}$ of the critical orbit $S^1\times\{0\}$,
the Bott integral $f$ equals $f(r)=c\pm r^2$, and up to an isotopic
deformation of the contact form in a neighbourhood
of the critical orbit---changing the Reeb flow but not the
Bott integral---the contact form is given by $\rmd\theta+
r^2\,\rmd\varphi$. Also, by Gray stability, the contact
structure $\xi=\ker\alpha$ remains unchanged, up to
a diffeotopy of~$M$, under this deformation.

In the present section we extend the arguments for achieving
those results in order to prove neighbourhood theorems
for tori and Klein bottles contained in $\Crit(f)$.
We then reason as in
\cite[Section~10.2]{ghs24} to remove those surfaces from the
critical set by a $C^{\infty}$-small perturbation
of the Bott integral~$f$, at the cost of introducing
an elliptic and a parallel hyperbolic Reeb orbit
into the critical set.

Taken together, these results will imply the following
genericity result.

\begin{thm}
Let $(M,\alpha)$ be a closed $3$-dimensional contact
manifold whose Reeb flow has a Bott integral~$f$. Then, after
an isotopic deformation of $\alpha$ and a $C^{\infty}$-small
deformation of $f$, both localised near the surfaces
contained in $\Crit(f)$, we may assume
that the Reeb flow has a Bott integral whose critical
set consists of isolated periodic Reeb orbits only.
\end{thm}
\subsection{Neighbourhood theorems for critical surfaces}
Since the Reeb vector field $R$ is tangent to the level sets
of a Bott integral~$f$, those level sets have to be periodic
Reeb orbits, tori, or Klein bottles. The vector field $Y$
on $M$ defined by
\begin{equation}
\label{eqn:Y}
\alpha(Y)=0,\;\;\; i_Y\rmd\alpha=-\rmd f,
\end{equation}
is non-singular and not collinear with~$R$ along
regular level sets, which in consequence have to be tori.
As critical level sets, all three types just mentioned can occur.
We now prove neighbourhood theorems for critical tori
and Klein bottles.

\begin{thm}[Neighbourhood theorem for critical tori]
\label{thm:critical-T}
Let $f$ be a Bott integral for the Reeb flow of $(M,\alpha)$,
and let $T^2\cong\Sigma\subset M$ be a component of a level set
$\{f=c\}$ along which the Bott integral has a local minimum
or maximum. Then a neighbourhood of $\Sigma$ in $M$ is diffeomorphic to
$[-\delta,\delta]\times T^2$ such that with coordinates $r\in[-\delta,\delta]$
and $x_1,x_2$ on $T^2=(\R/2\pi\Z)^2$, we have
\begin{enumerate}
\item[(i)] $f(r,x_1,x_2)=c\pm r^2$;
\item[(ii)] $\alpha=h_1(r)\,\rmd x_1+h_2(r)\,\rmd x_2$.
\end{enumerate}
\end{thm}

Recall that the contact condition for a contact form of Lutz type
as in (ii) is
\[ \Delta:=\left|\begin{array}{cc}
h_1 & h_1'\\
h_2 & h_2'
\end{array}\right|<0,\]
and the Reeb vector field
\[ R=\frac{h_2'\partial_{x_1}-h_1'\partial_{x_2}}{\Delta}\]
is then tangent to the $T^2$-factors.

\begin{proof}[Proof of Theorem~\ref{thm:critical-T}]
Statement (i) follows from the generalised Morse--Bott lem\-ma
\cite[Lemma~1.7]{bofo04}. In order to prove (ii), we slightly modify
the argument for proving \cite[Theorem~2.2]{ghs24}.
There we constructed covering maps $\R^2\rightarrow\{r\}\times T^2$
by using the flow of $R$ and~$Y$. In the present situation,
$Y$ vanishes along the critical set $\{r=0\}$. We circumvent this
problem by working with the vector field $Y'$ defined by
\[ \alpha(Y')=0,\;\;\; i_Y\rmd\alpha=\mp 2\,\rmd r;\]
in other words, we replace $\rmd f$ in the defining equations
\eqref{eqn:Y} for $Y$ by $\rmd f/r=\pm2\,\rmd r$ for $r\neq 0$.
This is well defined also in $r=0$, since $r$ is a smooth coordinate
on $(-\delta,\delta)$. So $Y'$ coincides with $Y/r$ for $r\neq 0$
and extends smoothly as a non-vanishing vector field over $r=0$.
The remainder of the argument is now as in the proof
of \cite[Theorem~2.2]{ghs24}.
\end{proof}

Next we describe the neighbourhood of a critical Klein bottle~$K$.
Recall from \cite{geth23} or \cite[Section~10.1]{ghs24}
that a closed tubular neighbourhood $\nu K$ of a Klein bottle $K$
embedded in an orientable $3$-manifold, which is the unique
non-trivial interval bundle over~$K$, can be described by
\[ \nu K=\bigl([-\delta,\delta]\times\R\times S^1\bigr)/
(r,s,\theta)\sim(-r,s-1,-\theta).\]
We write the class of $(r,s,\theta)$ as $[r,s,\theta]$.

In \cite{ghs24} we already worked with the following local model
as an example, but we did not provide a proof that this model is
universal. In fact, it is only universal up to a rescaling
of the contact form and the Bott integral.

\begin{thm}[Neighbourhood theorem for critical Klein bottles]
\label{thm:critical-K}
Let $f$ be a Bott integral for the Reeb flow of $(M,\alpha)$,
and let $K\cong\Sigma\subset M$ be a component of a level set
$\{f=c\}$ along which the Bott integral has a local minimum
or maximum. Then, up to a rescaling of the Bott integral,
as well as an isotopy and a rescaling of the
contact form, all supported in a neighbourhood of $\Sigma$,
a smaller neighbourhood of $\Sigma$ in $M$ looks like $\nu K$
with
\begin{enumerate}
\item[(i)] $f\bigl([r,s,\theta]\bigr)=c\pm r^2$;
\item[(ii)] $\alpha=\rmd s -r\,\rmd\theta$.
\end{enumerate}
\end{thm}

\begin{proof}
Again, statement (i) follows from the generalised Morse--Bott lemma.

The orientable double cover of $\nu K$ is given by
\[ [-\delta,\delta]\times T^2=
\bigl([-\delta,\delta]\times\R\times S^1\bigr)/
(r,s,\theta)\sim(r,s-2,\theta).\]
The contact form $\alpha$ and the Bott integral $f$ lift to this
double cover. In this lift, $\{0\}\times T^2=\{r=0\}$ is a component of
the critical set. Now, we cannot directly quote the previous result,
since we need to have a $\Z_2$-equivariant version of that
neighbourhood theorem, and we want to achieve a more
specific Lutz form. So we need to dig a little deeper into the
proof of the neighbourhood theorem for elliptic critical
orbits \cite[Theorem~2.4]{ghs24}, which contains the relevant ideas.
At each step in the following argument, we deal with
a $\Z_2$-equivariant isotopy or a $\Z_2$-invariant rescaling of the
contact form, all of which are supported in a neighbourhood
of~$\{r=0\}$. So the actual normal form we finally arrive at is
defined on a neighbourhood of $K$ defined by a smaller
$\delta$ than the one we started with.

First of all, as before one shows that the lifted contact form
$\talpha$ can be written as
\[ \talpha=h_0(r)\,\rmd r+h_1(r)\,\rmd s
+h_2(r)\,\rmd\theta.\]
Notice that the Reeb vector field $\wtR$ of $\talpha$
is tangent to the tori $\{r=\text{const.}\}$.
(In the proof of the previous theorem, the coordinates $x_1,x_2$ on the
torus factor are determined by a choice of basis of the lattice defining
the covering $\R^2\rightarrow T^2$; by applying a suitable element
of $\SL(2,\Z)$, we may assume these coordinates to be $(s,\theta)\in
(\R/2\Z)\times S^1$.)

The invariance under the $\Z_2$-action
\[ (r,s,\theta)\longmapsto (-r,s-1,-\theta) \]
on $[-\delta,\delta]\times T^2$ translates into
\[ h_0(-r)=-h_0(r),\;\;\; h_1(-r)=h_1(r),\;\;\;
\text{and}\;\;\; h_2(-r)=-h_2(r).\]
In particular, we have $h_0(0)=0$. As in Step~1 of the
proof of \cite[Theorem~2.4]{ghs24}, we first deform $h_0$,
relative to a neighbourhood of $\{\pm\delta\}\times T^2$,
so that it becomes identically zero near $r=0$; this can be done
within the class of odd functions. The convex linear
interpolation between the old and the deformed $h_0$
gives rise to a strict contact isotopy, defined by a flow
along the level surfaces $\{r=\text{const.}\}$, and stationary
on $\{r=0\}$, just like in the proof of
\cite[Theorem~2.2]{ghs24}; see equation (3) in that proof
and the argument following it.

The $\Z_2$-invariance of $\wtR$, which depends
on $r$ only, forces the Reeb vector field
to be a multiple of $\partial_t$ along $\{r=0\}$. This implies
that $h_1(0)\neq 0$, and by choosing the signs of the
coordinates $r,s$ appropriately,
we may assume that $h_1(0)>0$ and that the ambient
orientation is given by $\rmd r\wedge\rmd s\wedge\rmd\theta$. Then as
in Step~2 of the proof of \cite[Theorem~2.4]{ghs24}, we can make a
further deformation such that the contact form looks like
$\rmd s +h_2(r)\,\rmd\theta$ (with a modified, but still odd~$h_2$)
near $\{r=0\}$: for this we need to divide the contact
form by an even function $\chi(r)$ equal to
$h_1(r)$ near $r=0$, and identically equal to $1$
outside a small neighbourhood of $r=0$; so the
contact structure remains unchanged.
This deformation changes the Reeb vector field,
keeping it tangent to the tori $\{r=\text{const.}\}$.
Thanks to the contact condition, $h_2$ satisfies
$h_2'<0$.

Finally, as Step~3, we choose an odd function $h_2^*(r)$
(with negative derivative)
that equals $-r$ near $r=0$, and $h_2$ outside a small
neighbourhood of $r=0$. Then the convex linear
interpolation between $h_2$ and $h_2^*$ defines
a $\Z_2$-invariant deformation of $\talpha$ into the desired normal
form, which descends to the $\Z_2$-quotient $\nu K$.
One may check that this deformation of contact forms is induced by
the flow of a vector field $X=a(r)\partial_r$, with $a(r)$ an odd function.
In particular, this flow is stationary on $\{r=0\}$
and sends level sets of the Bott integral to level sets.
\end{proof}
\subsection{Removing critical surfaces}
\label{subsection:remove}
For simplicity of notation, we assume that we are dealing
with a critical torus or Klein bottle along which $f$
has an isolated local minimum, and that the value
of $f$ at the minimum is~$0$. In other words,
we are dealing with $f(r,x_1,x_2)=r^2$ in the case
of the torus, and with $f\bigl([r,s,\theta])=r^2$
in the case of a Klein bottle.

For the Klein bottle, we explained in \cite[Section~10.2]{ghs24}
how to perturb the Bott integral $f$ into one that has
two isolated critical Reeb orbits, one elliptic and one
hyperbolic, in place of the critical Klein bottle.

We briefly recall the construction for the Klein bottle,
and then show how to adapt it to the torus case.
We choose $\varepsilon >0$ such that $2\varepsilon<\delta$.
Let $\chi\co [-\delta,\delta]\rightarrow[0,\varepsilon^2]$
be a smooth function identically equal to $\varepsilon^2$
on the interval $[-\varepsilon,\varepsilon]$, and identically
equal to $0$ on the intervals $[-\delta,-2\varepsilon]$
and $[2\varepsilon,\delta]$. We may assume that
$|\chi'(r)|<|2r|$.

Recall that the Reeb vector field $R$ equals $\partial_s$
on the neighbourhood $\nu K$ of the critical
Klein bottle, so the function
\[ [r,s,\theta]\longmapsto r^2+\chi(r)\cos\theta\]
is still an integral of~$R$. A straightforward
computation, see~\cite{ghs24}, shows that this function
is a Bott integral with an elliptic critical orbit
at $(r,\theta)=(0,\pi)$, and a critical
hyperbolic orbit at $(r,\theta)=(0,0)$.

In the torus case, we modify the contact form
$\alpha=h_1(r)\,\rmd x_1+h_2(r)\,\rmd x_2$ near
$\{r=0\}$ in a way described in the proof of
\cite[Proposition~2.6]{ghs24}. By a small reparametrisation
of the curve $r\mapsto(h_1(r),h_2(r))$ we may assume that
$h_1(0)$ and $h_2(0)$ are rationally dependent.
Then, up to an $\SL(2,\Z)$-transformation
we may assume that $h_1(0)>0$ and $h_2(0)=0$,
which implies $h_2'(0)<0$. After a small deformation
and a further reparametrisation of the curve $r\mapsto(h_1(r),h_2(r)$,
up to a global constant we may finally assume that $h_1$ is identically
equal to $1$ near $r=0$, and $h_2(r)=-r$ near $r=0$. These modifications
change the Reeb vector field, keeping it
tangent to the tori $\{r=\text{const.}\}$, so $f$ is still a Bott integral.
Moreover, near $\{r=0\}$ the new Reeb vector field
equals~$\partial_{x_1}$, so the new Bott integral
\[ (r,x_1,x_2)\longmapsto r^2+\chi(r)\cos x_2\]
has the properties analogous to the previous case.
\section{The Euler class}
\label{section:Euler}
In this section we prove the easier direction of Theorem~\ref{thm:main}.
Given a contact form~$\alpha$ with Bott integral $f$ on a closed
$3$-manifold~$M$, we write $L_f\subset\Crit(f)$ for the link made up
of the critical periodic Reeb orbits. We show that the Poincar\'e dual
of the Euler class of the contact structure $\ker\alpha$ is represented
by~$L_f$, and that $L_f\subset M$ is a graph link.
\subsection{Euler class and critical orbits}
For the case that $L_f=\Crit(f)$, that is, in the absence of
critical surfaces, the following result was stated implicitly in the proof of
\cite[Lemma~7.2]{ghs24}. Here and throughout this paper, (co-)homology
groups are understood to be with $\Z$-coefficients.

\begin{lem}
\label{lem:Euler-formula}
Let $f$ be a Bott integral for $(M,\alpha)$. Write $E\subset L_f$ for the
set of elliptic orbits, and $H\subset L_f$ for the set of hyperbolic
orbits. We orient the critical orbits by the Reeb flow.
Then the Euler class of $\xi=\ker\alpha$ is given by
\[ e(\xi)=\sum_{\gamma\in E}\PD([\gamma])-\sum_{\gamma\in H}
\PD([\gamma])\in H^2(M).\]
\end{lem}

\begin{proof}
The vector field $Y$ defined in \eqref{eqn:Y} can be read as a
section $Y\co M\rightarrow\xi$ of the $2$-plane bundle $\xi$. The
intersection of $Y(M)$ with the zero section $0_{\xi}$ of $\xi$ equals
$\Crit(f)$.

(i) We first consider the case that there are no critical surfaces.
From the normal form of the Bott integral $f$ in terms of
transverse coordinates $x,y$ near critical orbits, $\pm(x^2+y^2)$
near elliptic orbits and $xy$ near hyperbolic ones, it follows that
$Y(M)$ intersects $0_{\xi}$ transversely. Both $Y(M)$ and $0_{\xi}$
carry the orientation induced by~$M$, and $\xi$ the orientation
given by~$\rmd\alpha$. This defines an orientation on
$L_f=Y(M)\cap 0_{\xi}$. As in the $2$-dimensional Poincar\'e--Hopf
theorem one then sees that for $\gamma\in E$, this orientation coincides
with the direction of the Reeb flow; for $\gamma\in H$ it is the
opposite orientation. Regarding this last statement, observe that the
$R$-invariance of $f$ allows us to reduce the reasoning to the
consideration of a local surface of section of the Reeb flow
near the critical Reeb orbit, oriented by the area form~$\rmd\alpha$.

(ii) Now suppose that $\Crit(f)$ does contain critical surfaces. As we have
seen in Section~\ref{section:crit}, any surface contained in $\Crit(f)$ can be
replaced by a pair of mutually isotopic critical Reeb orbits, one elliptic
and one hyperbolic, so in sum they do not contribute to the Euler class.
This means that the formula for $e(\xi)$ remains true in this general case.
\end{proof}
\subsection{The link of critical orbits}
With the next lemma we complete the proof of the `only if' part
of Theorem~\ref{thm:main}. Note that this is a general statement
about Bott-integrable contact forms; overtwistedness of the
contact structure is not needed here.

\begin{lem}
\label{lem:Lf-graph}
Let $f$ be a Bott integral for $(M,\alpha)$. Then $L_f\subset M$ is a
graph link.
\end{lem}

\begin{proof}
In \cite[Section~3.2]{ghs24} we gave a summary of the results
from \cite{bofo04} concerning the (singular) Liouville foliation defined by
the level sets of a Bott integral. The key statement is that
up to a diffeomorphism preserving the Liouville foliation,
any Bott-integrable Reeb flow can be constructed from some simple building
blocks, which are products of a surface (with boundary) with~$S^1$,
or Seifert bundles with singular fibres of order~$2$, coming from
fixed points of an involution on the surface. The Liouville foliation
on these building blocks is defined by the level sets of a Morse--Bott
function lifted from a Morse--Bott function on the base, with all
fixed points of the involution contained in the critical set of that
Morse function.

From this description of the building blocks one sees that the complement
of tubular neighbourhoods of the critical orbits is fibred by circles.
So the exterior of~$L_f$, which is the the manifold obtained by gluing
these complements, is a graph manifold.
\end{proof}

\begin{rem}
In \cite[Section~3.2.3]{ghs24}, we wrote that the Morse--Bott functions
on the Seifert fibred building blocks are lifted from Morse functions
on the base, when in fact we should have written "Morse--Bott functions
on the base". The neighbourhood $\nu K$ of a critical Klein bottle is
a case in point. As described in detail in~\cite{geth23}, $\nu K$
may be thought of as the Seifert bundle over a disc with two singular fibres
of order~$2$. This corresponds to starting with the annulus
$[-\delta,\delta]\times S^1$ with Morse--Bott function
$f(r,\theta)=r^2$ and involution $\tau(r,\theta)=(-r,-\theta)$
having two fixed points. One then obtains $\nu K$ with the described Seifert
fibration as the mapping torus of~$\tau$.

Critical tori may also be produced by the gluing of the building blocks
along torus boundaries.
\end{rem}
\section{Seifert manifolds}
\label{section:Seifert}
In \cite[Theorem~1.8]{ghs24} we showed that every contact structure
on a closed Seifert manifold invariant under the fixed-point free
$S^1$-action defining the Seifert fibration is Bott integrable.
In the present section we deal with overtwisted (but not necessarily
$S^1$-invariant) contact structures on Seifert manifolds.

Our aim will be to find a Bott-integrable contact structure on any given
Seifert manifold that is sufficiently `rich' in the following sense.
When we speak of a Bott-integrable contact structure on a manifold
with boundary, it is understood that the boundary components are tori,
and the contact structure is of Lutz type on collar neighbourhoods
of those boundary components. In particular, the boundary
is a regular level set of the Bott integral, and the Reeb flow
is tangent to the boundary. By the sewing lemma \cite[Lemma~3.4]{ghs24},
the gluing along such boundaries produces again a Bott-integrable
contact structure.

\begin{defn}
A Bott-integrable contact structure $\xi$ on a compact $3$-manifold~$M$
is called \textbf{$H_1$-complete} if, given any
homology class $u\in H_1(M)$, one can find a a Bott-integrable
contact form $\alpha$ defining $\xi$ with a Bott integral $f$
whose critical link $L_f$ contains a sublink of elliptic periodic
Reeb orbits representing~$u$.
\end{defn}

The methods developed in \cite{ghs24} allow us to prove the
following key lemma.

\begin{lem}
\label{lem:H1-ot}
If a closed $3$-manifold $M$ admits an $H_1$-complete Bott-integrable contact
structure, then every overtwisted contact structure on $M$
is Bott integrable.
\end{lem}

\begin{proof}
By Eliashberg's classification \cite{elia89} of overtwisted contact structures
we need only show that every homotopy class of tangent $2$-plane fields
on $M$ contains a Bott-integrable overtwisted contact structure. Write
$\xi_0$ for the given $H_1$-complete Bott-integrable contact structure
on~$M$. We may assume that $\xi_0$ is overtwisted; if it is not, we
create an elliptic critical orbit (together with
a hyperbolic one) as in \cite[Section~2.3]{ghs24}, and then perform
a full Lutz twist, which does not change the homotopy class
of $\xi_0$ as a tangent $2$-plane field.

Now, given any tangent $2$-plane field~$\eta$, there is
an obstruction $d^2(\xi_0,\eta)\in H^2(M)$ for $\xi_0$ to
be homotopic to $\eta$ over the $2$-skeleton of~$M$. This obstruction
is antisymmetric and additive; see \cite[Section~4.2]{geig08}:
\[ d^2(\eta_1,\eta_2)=-d^2(\eta_2,\eta_1)
\;\;\;\text{and}\;\;\;
d^2(\eta_1,\eta_2)+d^2(\eta_2,\eta_3)=d^2(\eta_1,\eta_3).\]
The relation of the obstruction class $d^2$ with the Euler class
is given by
\[ 2d^2(\eta_1,\eta_2)=e(\eta_1)-e(\eta_2);\]
see \cite[Remark~4.3.4]{geig08}.
Thus, in the presence of $2$-torsion in $H^2(M)$ the obstruction
class is not detected by the Euler class.

If a contact structure $\xi$ is modified by a Lutz twist
along a positively transverse knot $K\subset (M,\xi)$ into a new
contact structure~$\xi^K$, then $d^2(\xi^K,\xi)=-\PD[K]$.

The $H_1$-completeness of $\xi_0$ allows us to choose a Bott-integrable
contact form $\alpha_0^u$ with a Bott integral $f_u$ whose
critical link $L_{f_u}$ contains a sublink $L_u$ of elliptic Reeb
orbits representing the class $u:=\PD(d^2(\xi_0,\eta))\in H_1(M)$.
By performing Lutz twists along the components of~$L_u$, we obtain
a Bott-integrable contact form $\alpha_1^u$ (with the same
Bott integral~$f_u$) defining a contact structure
$\xi^u=\ker\alpha_1^u$ with $d^2(\xi^u,\xi_0)=-\PD(u)$, and hence
$d^2(\xi^u,\eta)=0$.

Finally, in order to modify $\xi^u$ into a Bott-integrable
contact structure homotopic to~$\eta$ over all of~$M$,
it suffices to take a contact connected sum as in \cite[Section~6]{ghs24}
with a suitable overtwisted contact structure on~$S^3$, all of which
are Bott integrable by \cite[Theorem~1.9]{ghs24}.
\end{proof}

Here is the central result of this section.

\begin{prop}
\label{prop:Seifert-H1}
Every compact Seifert manifold (possibly with toral boundaries)
admits an $H_1$-complete Bott-integrable contact structure.
\end{prop}

The following corollary is then immediate from Lemma~\ref{lem:H1-ot}
and Proposition~\ref{prop:Seifert-H1}.

\begin{cor}
\label{cor:Seifert-ot}
Every overtwisted contact structure on any closed
Seifert manifold is Bott integrable.\qed
\end{cor}

We now turn to the proof of Proposition~\ref{prop:Seifert-H1}.
Most of the work goes into proving this result for
a product bundle $\Sigma\times S^1$, where $\Sigma$ is a compact orientable
surface. The case of general Seifert bundles over orientable surfaces
then follows without much effort. The case of a Seifert bundle
over a non-orientable base (with orientable total space) requires
an extra argument involving a construction of fibre connected sums
for Bott-integrable contact structures.
\subsection{Orientable base}
We start by looking at the case of product bundles $\Sigma\times S^1$
with an orientable base surface~$\Sigma$, which is allowed to
have boundary. We recall from \cite{ghs24} how
to construct a Bott-integrable contact form on $\Sigma\times S^1$
from an open book decomposition. We then modify this contact form
by introducing homotopically inessential Lutz components, and turn some
periodic Reeb orbits inside these Lutz components into critical ones
by changing the Bott integral.
\subsubsection{Open books}
As discussed in \cite[Section~7]{ghs24}, the manifold
$\Sigma\times S^1$ may be thought of as an open book with page
$\Sigma_0$ equal to $\Sigma$ with two open discs removed, and monodromy
a couple of Dehn twists (one right-handed, one left-handed)
parallel to the two boundary components $\partial\Sigma_0\setminus
\partial\Sigma$ created by the removal of these discs.
For the construction of an $H_1$-complete Bott-integrable contact structure
on $\Sigma\times S^1$, we need the following mild extension
of \cite[Lemma~3.3]{ghs24}. We formulate it for a single oriented,
embedded circle $\gamma\subset\Sigma_0$, but it works equally well for a
finite collection of pairwise disjoint embedded circles.

We choose $1$-forms $\lambda_i=\rho_i(r_i)\,\rmd\varphi_i$, with $\rho_i'>0$,
on collar neighbourhoods of the finitely many boundary components
$\partial_i\Sigma_0$, using collar coordinates $(r_i,\varphi_i)
\in(-1,0]\times S^1$. Furthermore---this is where we extend the
statement from our previous paper---we consider the $1$-form
$\lambda_{\gamma}=s\,\rmd\varphi$ in terms of coordinates
$(s,\varphi)\in(-1,1)\times S^1$ on a neighbourhood of~$\gamma\equiv
\{0\}\times S^1$. It will be understood that this neighbourhood of $\gamma$
is disjoint from the support of the Dehn twists.

\begin{lem}
If $\sum_i\rho_i(0)>0$, there is an exact area
form $\omega=\rmd\lambda$ on $\Sigma_0$ with $\lambda=\lambda_i$ near
$\partial_i\Sigma$ and $\lambda=\lambda_{\gamma}$ near~$\gamma$.
\end{lem}

\begin{proof}
Thanks to the assumption in the lemma, for $\varepsilon>0$
sufficiently small we still have $\sum_i\rho_i(-\varepsilon)>0$.
Hence
\[ \sum_i\int_{\partial_i\Sigma_0}\lambda_i=2\pi\sum_i\rho_i(0)>
2\pi\sum_i\bigl(\rho_i(0)-\rho_i(-\varepsilon)\bigr)=
\sum_i\int_{(-\varepsilon,0]\times\partial_i\Sigma_0}\rmd\lambda_i.\]
This allows us to choose an area form
$\omega$ on $\Sigma_0$ that coincides with $\rmd\lambda_i$ near
$\partial_i\Sigma_0$ and with $\rmd\lambda_{\gamma}$ near~$\gamma$, and which
satisfies
\[ \int_{\Sigma_0}\omega=\sum_i\int_{\partial_i\Sigma_0}\lambda_i.\]
Now apply the argument from \cite[Lemma~3.3]{ghs24}, involving
compactly supported cohomology, to the surface obtained by cutting open
$\Sigma_0$ along~$\gamma$ in order to find a primitive $\lambda$
of $\omega$ with the desired properties.
\end{proof}

Write $\psi\co\Sigma_0\rightarrow\Sigma_0$ for the diffeomorphism
given by the two boundary-parallel Dehn twists.
As shown in \cite[Section~7.2]{ghs24}, an area form $\omega=\rmd\lambda$
as just constructed gives rise to a Bott-integrable contact form
on $\Sigma\times S^1$, where the Bott integral on the mapping torus
$\Sigma_0(\psi)$ can be taken to be one induced from any Morse
function on $\Sigma_0$ that only depends on the radial coordinate
in the collar neighbourhoods of the boundary (and hence is invariant
under~$\psi$). On the complement $\Sigma_0'\subset\Sigma_0$ of two collars
containing the annuli where the Dehn twists are performed,
the mapping torus looks like a product $\Sigma_0'\times S^1$ with
contact form $\lambda+C\,\rmd\theta$, where $C$ is some
large positive constant.

The boundary of this mapping torus is
$\partial\Sigma_0(\psi)=\partial\Sigma_0\times S^1$,
and the manifold $\Sigma\times S^1$ is obtained by gluing
two solid tori $S^1\times D^2$ to this mapping torus along
the boundary components $(\partial\Sigma_0\setminus\partial\Sigma)\times S^1$.
Then extend the Morse--Bott function on $\Sigma_0(\psi)$ over these two
solid tori as a function in the radial coordinate on the $D^2$-factor,
with an isolated non-degenerate maximum at the centre. This Morse--Bott
function is an integral for the Giroux contact form on the open book
decomposition of $\Sigma\times S^1$.

The next lemma says that the Morse function on $\Sigma_0$ can be chosen
to be adapted to any embedded circle $\gamma\subset\Sigma_0$
(or a collection of such circles).

\begin{lem}
Given $\gamma\subset\Sigma_0$, we can find a Morse function on $\Sigma_0$
which is a strictly monotone function of the radial coordinate in the
collar neighbourhoods of the boundary, and a strictly monotone function of
$s$ in the coordinates $(s,\varphi)\in (-1,1)\times S^1$ around~$\gamma$.
\end{lem}

\begin{proof}
Let $\Sigma_0^{\gamma}$ be the surface obtained by cutting open $\Sigma_0$
along $\gamma$ and compactifying it by including the new boundary
circles~$\pm\gamma$. The desired Morse function is given by the
height function (say, the $z$-coordinate) on a copy of $\Sigma_0^{\gamma}$
embedded in $\R^3$ in such a way that each boundary component of
$\Sigma_0^{\gamma}$ lies in a plane $\{z=\text{const.}\}$,
with the boundary components $\pm\gamma$ placed at the same height,
one as a lower and one as an upper boundary. All other boundary components
may be chosen to be local maxima of the height function.
\end{proof}

\begin{Note}
(1) The isolated minima of the Morse function on $\Sigma_0$ give rise
to elliptic critical Reeb orbits. By performing a full Lutz twist
along such an orbit, we obtain an overtwisted Bott-integrable
contact structure in the same homotopy class of tangent $2$-plane fields.

(2) We may assume that the $1$-form $\lambda$ we construct for any given
simple closed loop $\gamma\subset\Sigma_0$ does not depend on $\gamma$
near~$\partial\Sigma_0$. Then the convex linear interpolation between
the $1$-forms $\lambda,\lambda'$ corresponding to two curves
$\gamma,\gamma'\subset\Sigma_0$ induces a homotopy of contact forms relative
to a neighbourhood of~$\partial\Sigma_0$. This means that the
isotopy class of the contact structure we constructed on $\Sigma\times S^1$
does not depend on~$\gamma$.
\end{Note}
\subsubsection{Introducing Lutz components and critical orbits}
\label{subsubsection:Lutz+critical}
Let $\gamma\subset\Sigma_0$, its neighbourhood $(-1,1)_s\times S^1_{\varphi}$,
and the adapted Morse function on $\Sigma_0$ be as before.
On the thickened torus $(-\varepsilon,\varepsilon)_s\times
S^1_{\varphi}\times S^1_{\theta}$, for $\varepsilon>0$ sufficiently small,
the contact form $\lambda+C\,\rmd\theta$ equals $s\,\rmd\varphi+C\,\rmd\theta$.
We now perform a full Lutz twist, that is, we
replace this contact form by $h_1(s)\rmd\varphi+h_2(s)\rmd\theta$,
with $(-\varepsilon,\varepsilon)\ni s\mapsto (h_1(s),h_2(s))$ describing the
planar curve as shown in Figure~\ref{figure:full-Lutz}, subject to the
boundary conditions $h_1(s)=s$ and $h_2(s)=C$ near $s=\pm\varepsilon$.
In this way we construct three smaller thickened tori where, up to irrelevant
positive constants, the contact form is given by $h_1(s)\,\rmd\varphi\pm
\rmd\theta$ or $\rmd\varphi+h_2(s)\rmd\theta$, whose Reeb vector
field is a constant positive multiple of $\pm\partial_{\theta}$
or $\partial_{\varphi}$, respectively. 

\begin{figure}[h]
\labellist
\small\hair 2pt
\pinlabel $h_1$ [t] at 427 217
\pinlabel $h_2$ [r] at 216 430
\pinlabel $C$ [tr] at 215 395
\pinlabel $\varepsilon$ [t] at 324 206
\pinlabel $-\varepsilon$ [tr] at 114 206
\endlabellist
\centering
\includegraphics[scale=0.4]{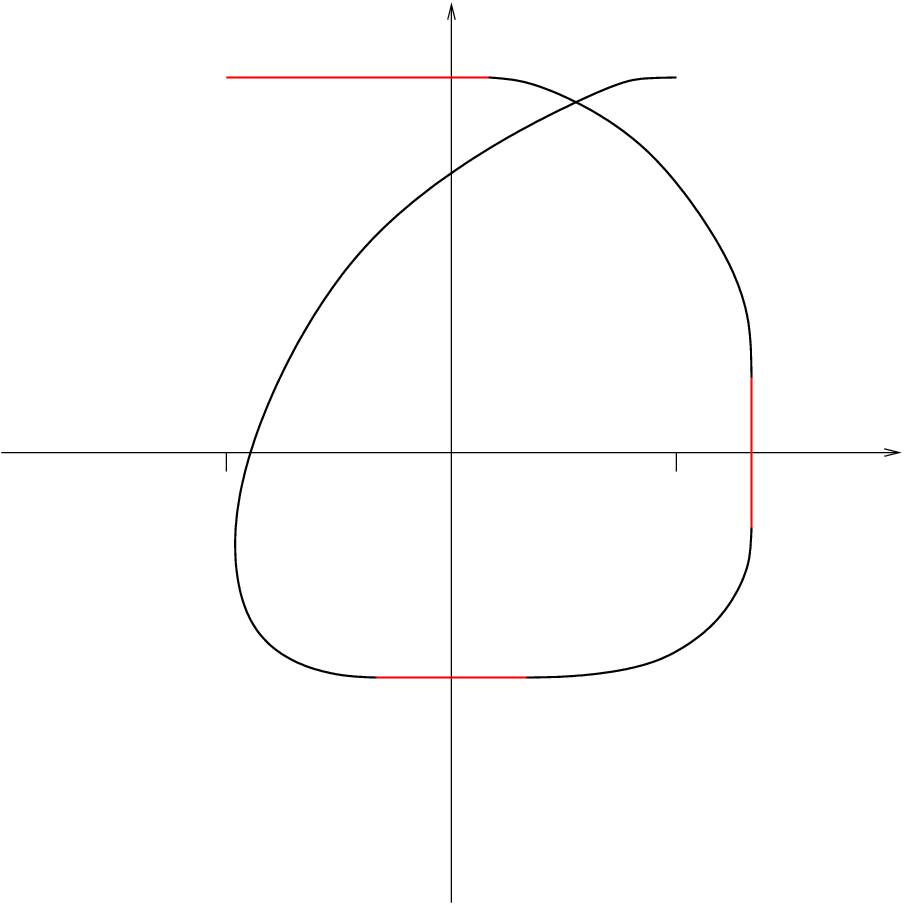}
  \caption{Introducing a Lutz component}
  \label{figure:full-Lutz}
\end{figure}
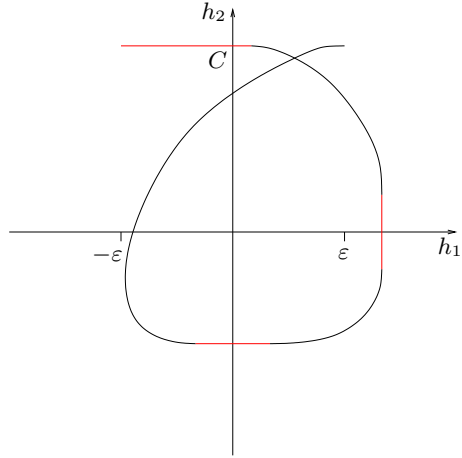

This modified contact form still has the same Bott integral
(a monotone function of $s$), and the homotopy class of the
contact structure as a tangent $2$-plane field
is unchanged by the full Lutz twist; see
\cite[Lemma~4.5.3]{geig08}.

In a second step, we modify the Bott integral by a finite number
of `reverse Giroux eliminations' in the $(s,\varphi)$- or the
$(s,\theta)$-plane transverse to the Reeb flow in the direction
of $\pm\partial_{\theta}$ or $\partial_{\varphi}$, respectively;
see \cite[Section~2.3]{ghs24} for the details. In this way,
we can introduce any finite number of pairs of elliptic and
hyperbolic critical orbits into the Reeb flow, lying in the
isotopy class of $\pm S^1_{\theta}$ (the fibre class) or $S^1_{\varphi}$
(the class represented by~$\gamma$).
\subsubsection{Proof of Proposition~\ref{prop:Seifert-H1} for orientable base}
Any Seifert manifold $M$ with orientable base is obtained from a
product manifold $\Sigma\times S^1$, where $\Sigma$ may always be assumed
to have at least one boundary component, by gluing solid tori to some
of the boundary components. The contact form we constructed
on $\Sigma\times S^1$ is of Lutz type near the boundary, and so it extends
over the solid tori as a Bott-integrable contact form with an
elliptic critical Reeb orbit at each spine.

Write $g$ for the genus of~$\Sigma$, and $n+1\geq 1$ for the number of
boundary components, so that $\Sigma$ is obtained from the closed,
orientable surface $\Sigma_g$ of genus $g$ by removing $n+1$ discs.
The homology group $H_1(\Sigma\times S^1)\cong\Z^{2g+n+1}$ is
freely generated by the standard generators $a_1,b_1,\ldots,a_g,b_g$ of
$H_1(\Sigma_g)$, boundary parallel loops $q_1,\ldots,q_n$ to $n$
of the $n+1$ boundary components, and the
fibre class $h=[*\times S^1]$. The gluing of solid tori
adds relations to give us $H_1(M)$, but no new generators.
It therefore suffices to show that the contact structure on
$\Sigma\times S^1$ is $H_1$-complete.

Any class $u\in H_1(\Sigma\times S^1)$ can be written as
\[ u=kc+\sum_{i=1}^n\ell_iq_i+mh\]
with $c$ an indivisible class in the subgroup
generated by $a_1,b_1,\ldots, a_g,b_g$, the coefficient $k$
a non-negative integer, and $\ell_1,\ldots,\ell_n,m\in\Z$.
By \cite{mepa78}, the class $c$ can be represented by a simple closed
loop $\gamma \subset\Sigma_0$. The sum $\sum_i\ell_iq_i$
can be represented by a disjoint union of boundary parallel curves
in the collars of $\partial\Sigma\subset\Sigma_0$, where these
curves are level sets of the Bott integral.
The class $mh$ is represented by $m$ parallel $S^1$-fibres of
$\Sigma_0\times S^1$. Thus, our arguments
above show that $u$ can be represented by a collection of some
elliptic critical Reeb orbits of a suitably modified Bott-integrable
contact form on $\Sigma\times S^1$.

\begin{rem}
In order to obtain a Seifert manifold, the solid tori
have to be attached to $\Sigma\times S^1$ by sending the
meridian of the solid torus to a simple closed curve
on the boundary torus winding non-trivially around the
corresponding boundary component of~$\Sigma$. If instead the meridian is
glued to a curve representing the $S^1$-fibre of $\Sigma\times S^1$,
one speaks of a \emph{generalised} Seifert manifold; see~\cite{jane83}.
Our argument applies equally to those.
\end{rem}
\subsection{Non-orientable base}
Any Seifert manifold over a non-orientable base can be obtained
from a Seifert manifold over an orientable base by a fibre connected
sum (along a regular fibre) with the unique non-trivial $S^1$-bundle
over~$\RP^2$. When we think of $\RP^2$ as being obtained by gluing
a M\"obius band and a $2$-disc, the restriction of the non-trivial
$S^1$-bundle to the M\"obius band gives us precisely the Klein
bottle neighbourhood $\nu K$ described earlier. The only new homological
contribution comes from the spine of the M\"obius band.
\subsubsection{Modification near a Klein bottle}
On $\nu K$ we take the standard Bott-integrable contact form
$\alpha=\rmd s-r\,\rmd\theta$ as in Theorem~\ref{thm:critical-K}.
Inside this neighbourhood we can introduce a Lutz component
as in Section~\ref{subsubsection:Lutz+critical}
near a regular torus $\{r=r_0\}$ for some small positive~$r_0$,
and then introduce an elliptic critical Reeb orbit in the class
of the $S^1$-fibre. Also, by a perturbation as
in Section~\ref{subsection:remove}, we create an elliptic critical Reeb orbit
inside the Klein bottle---thought of as the non-orientable $S^1$-bundle over
the spine of the M\"obius band---in the class of the spine.

The latter elliptic orbit gives us the $H_1$-completeness;
the former allows us to perform the fibre connected sum.
\subsubsection{Fibre connected sum}
The general construction of fibre connected sums of contact
manifolds along two diffeomorphic contact submanifolds with
opposite normal bundles has been described in \cite[Section~7.4]{geig08}.

In our situation, an \emph{ad hoc} construction is much simpler
and yields the Bott-integrability to boot. Indeed, the fibre connected
sum along two elliptic critical Reeb orbits can be defined
by removing standard neighbourhoods as in \cite[Theorem~2.4]{ghs24}
around the two orbits, and then identifying the toral boundaries
of the complements by a diffeomorphism that sends the fibre circle
on one side to the fibre circle on the other, and glues the meridional circles
with a flip of the orientation. Notice that the outer normal on one
complement glues with the inner normal on the other, so the gluing of the
torus boundary should indeed be done with an orientation-reversing
diffeomorphism.

Since the contact form near the boundary of a tubular neighbourhood
of an elliptic orbit is of Lutz type, the Bott-integrable contact
form on the fibre connected sum can now be constructed by an
application of the sewing lemma \cite[Lemma~3.4]{ghs24}.

This concludes the proof of Proposition~\ref{prop:Seifert-H1}.
\section{Proof of the main theorem}
\label{section:proof}
We now want to prove the `if' part of Theorem~\ref{thm:main}
for a fixed closed, orientable $3$-dimensional graph manifold~$M$.
Our task is to show that any overtwisted contact structure
$\xi$ on $M$ with $\PD(e(\xi))\in H_1(M)$ representable by a graph link
is Bott integrable. Similar to the proof of this statement
for Seifert manifolds in Corollary~\ref{cor:Seifert-ot}, we need to find
a Bott-integrable overtwisted contact structure $\xi_0$
on~$M$ which is $H_1$-complete within the range of homology classes
representable by graph links.

This strategy is based on a homological characterisation of graph links,
which has been established by Yano~\cite{yano85b}. We recall
some of the relevant $3$-manifold topology; for more details we recommend
Hatcher's notes~\cite{hatc23}.
\subsection{Graph manifolds and the JSJ decomposition}
Graph manifolds, as introduced by Waldhausen~\cite{wald67I,wald67II},
are the closed, orientable $3$-manifolds that can be decomposed
along a disjoint collection of embedded $2$-tori into
Seifert fibred pieces.

Given a graph manifold $M$, we consider a prime decomposition
\[ M=M_1\#\cdots\# M_n\]
of~$M$ as the connected sum of prime $3$-manifolds $M_1,\ldots,M_n$.
Recall that $M_i$ being prime means that it is not diffeomorphic to
the $3$-sphere~$S^3$, and that in any connected sum decomposition
$M_i=M_i'\# M_i''$, one of the summands $M_i',M_i''$ must be~$S^3$.
In this prime decomposition, the summands $M_1,\ldots,M_n$
are uniquely determined by~$M$, but the isotopy class
of the collection of $2$-spheres defining the splitting into prime
summands is not.

The summands $M_1,\ldots, M_n$ in the prime decomposition of
a graph manifold $M$ are also graph manifolds; see
Theorems 6.3 and 7.1 in~\cite{wald67II}. Conversely, the
connected sum of graph manifolds is a graph manifold;
see the remark on page 91 of \cite{wald67II} or
\cite[Lemma~3.1]{ghs24}.

A $3$-manifold is called irreducible if every embedded $2$-sphere
bounds a ball. In particular, an irreducible $3$-manifold is prime.
The manifold $S^1\times S^2$ plays a
special role in $3$-manifold topology as the only prime manifold
that is not irreducible.

A $2$-torus embedded in a $3$-manifold is said to be incompressible
if the inclusion map induces a monomorphism on fundamental
groups. A $3$-manifold is called atoroidal if the only
incompressible tori it contains are boundary parallel.

If $M$ is an irreducible manifold, it can be cut along a disjoint
collection of incompressible tori into pieces that are either
atoroidal or Seifert fibred. Rather amazingly, the minimal
(with respect to inclusion) collection of such tori is unique
up to isotopy; see \cite[Theorem~1.9]{hatc23}. This unique decomposition
is known as the JSJ decomposition of~$M$, named after
Jaco--Shalen~\cite{jash79} and Johannson~\cite{joha79}.

It turns out that amongst irreducible $3$-manifolds, graph manifolds
are characterised as those whose JSJ decomposition only contains Seifert
fibred pieces. One direction of this statement is obvious; the other
is the content of \cite[Theorem~6.3]{wald67II}. In fact, in the
contemporary literature this characterisation is often taken as
the definition of a graph manifold.
\subsection{Graph links}
A graph link is a link in a closed, orientable $3$-manifold
whose exterior, that is, the complement of an open tubular neighbourhood,
is a graph manifold. Observe that any sublink of a graph link,
including the empty one, is graph; in particular, any manifold
containing a graph link is a graph manifold.
\subsubsection{Graph manifolds prime to $S^1\times S^2$}
A closed, orientable $3$-manifold $M$
is called prime to $S^1\times S^2$ if it does not contain any
summands $S^1\times S^2$ in its prime decomposition, so that $M$
is the connected sum of irreducible manifolds.

In \cite{yano85a,yano85b}, Yano associates with a graph manifold $M$
prime to $S^1\times S^2$ its JSJ complex $\calC_M$, which is a
$1$-dimensional complex built as follows. Write $M$ as the connected sum
$M=M_1\#\cdots\# M_n$ of irreducible manifolds $M_1,\ldots,M_n$. For each
$M_i$ we consider its JSJ decomposition into Seifert fibred pieces.
Thicken the incompressible tori
defining this decomposition into disjointly embedded copies of
$T^2\times[-1,1]$; the closure of
their complement is then a finite collection $\{S_{ij}\}$
of Seifert manifolds with toral boundaries.

Now, for each irreducible summand $M_i$ we build the $1$-complex
consisting of one vertex for each~$S_{ij}$ and
a connecting edge for the gluing of two boundary components.
In the JSJ decomposition it may well happen that one toral
boundary component of a Seifert fibred piece is glued to
another boundary component of the same piece, as we shall see in an
example in Section~\ref{section:cat}, so some $1$-cells in the
complex may form loops.

We then introduce $n-1$ further
edges, connecting a vertex in the graph for $M_i$ with one
in the graph for $M_{i+1}$, $i=1,\ldots,n-1$. The resulting
complex $\calC_M$ does not depend, up to homotopy equivalence, on
the choices in forming the prime decomposition of $M$
and in connecting the $n$ constituent parts.

There is a natural map $\rho\co M\rightarrow\calC_M$, defined
by collapsing the Seifert fibred pieces (possibly with
an open $3$-ball removed if the connected sum
is performed in that piece) of the prime summands
to the corresponding vertex, and the thickened tori as well
as the thickened $2$-spheres where the connected sums are formed
to the corresponding edge. In other words, the neighbourhoods
$T^2\times[-1,1]$ and $S^2\times[-1,1]$ of those tori and spheres
are projected onto the interval $[-1,1]$, regarded as the edge
connecting the vertices in $\calC_M$ representing the adjacent
Seifert fibred piece(s).

Here is the main theorem from~\cite{yano85b}.

\begin{thm}[Yano]
\label{thm:Yano1}
Let $M$ be a graph manifold prime to $S^1\times S^2$.
A homology class $u\in H_1(M)$ can be represented by a graph link
if and only if $\rho_*(u)=0$ in $H_1(\calC_M)$.\qed
\end{thm}

As a crucial preliminary step for this theorem, Yano establishes
the particular case that every element in the first homology
group of a Seifert manifold can be represented by a graph
link \cite[Proposition~2.2]{yano85b}. Our argument in
Section~\ref{section:Seifert} gives an
alternative proof of that part.
\subsubsection{Graph manifolds containing $(S^1\times S^2)$-summands}
If the graph manifold $M$ is not prime to $S^1\times S^2$, one
fixes a description of $M$ as a connected sum
\[ M=N\#k(S^1\times S^2),\]
with $N$ a graph manifold prime to $S^1\times S^2$.
This induces a splitting $H_1(M)=H_1(N)\oplus\Z^k$, and we write
elements of $H_1(M)$ as pairs $(a,b)$ with respect to this splitting.
For $b\neq 0$, we write $m(b)$ for the largest natural number
dividing $b\in\Z^k$, and we set $m(0)=0$.
The JSJ complex $\calC_N$ and the map
$\rho^N\co N\rightarrow\calC_N$ are constructed as before.

Here is Yano's homological characterisation of graph links
\cite[Theorem~3.3]{yano85b} in this general setting.

\begin{thm}[Yano]
\label{thm:Yano2}
Let $M=N\#k(S^1\times S^2)$ be a graph manifold, with $N$
prime to $S^1\times S^2$.
A homology class $u=(a,b)\in H_1(M)=H_1(N)\oplus\Z^k$ can be represented
by a graph link if and only if $\rho^N_*(a)$ is divisible by $m(b)$.\qed
\end{thm}
\subsection{Graph manifolds prime to $S^1\times S^2$}
We can now prove the `if' part of Theorem~\ref{thm:main} for
any graph manifold $M$ prime to $S^1\times S^2$. Write
$M$ as a connected sum $M=M_1\#\cdots\# M_n$ of irreducible
graph manifolds. Construct an overtwisted Bott-integrable
contact structure $\xi_0$ on $M$ as follows:
\begin{itemize}
\item[(i)] First we construct a Bott-integrable contact structure
on each $M_i$ by sewing together $H_1$-complete contact structures
on the Seifert pieces $S_{ij}$ of the JSJ decomposition of~$M_i$.
\item[(ii)] Form the connected sum of the $M_i$ as in \cite[Section~6]{ghs24}.
\item[(ii)] If necessary, perform a full Lutz twist along an
elliptic critical orbit to make the contact structure overtwisted.
\end{itemize}

From Section~\ref{section:Euler} we know that the Poincar\'e
dual of the Euler class $e(\xi_0)$ is represented by a graph link.
Hence, by Theorem~\ref{thm:Yano1}, we know that $\rhod(e(\xi_0))=0$,
where for ease of notation we set
\[ \rhod:=\rho_*\circ\PD\co H^2(M)\longrightarrow H_1(\calC_M).\]

Now let $\eta$ be any tangent $2$-plane field on $M$ with an Euler class
$e(\eta)$ whose Poincar\'e dual can be represented by a graph link, so that
$\rhod(e(\eta))=0$. Then for the obstruction class $d^2(\xi_0,\eta)$
we find
\[ 2\rhod\bigl(d^2(\xi_0,\eta)\bigr)=
\rhod\bigl((e(\xi_0)-e(\eta)\bigr)=0.\]
Since $H_1(\calC_M)$ is a free abelian group, this implies
$\rhod(d^2(\xi_0,\eta))=0$.

The class $\PD(d^2(\xi_0,\eta))$ can be represented by an oriented
link $L_{\eta}\subset M$ intersecting the neighbourhoods $S^2\times[-1,1]$
and $T^2\times[-1,1]$ of the spheres defining the connected sum and
the tori defining the JSJ decomposition, respectively, along
intervals $*\times[-1,1]$, each of which is mapped
to an edge of~$\calC_M$. The fact that the class
$\rhod(d^2(\xi_0,\eta))=[\rho(L_{\eta})]\in H_1(\calC_M)$ vanishes implies that 
the intersection number of the $1$-cycle $\rho(L_\eta)\subset\calC_M$ with
any point in $\calC_M$ is zero. For $L_{\eta}\subset M$ this means that
the link traverses each neighbourhood $S^2\times[-1,1]$ or
$T^2\times[-1,1]$ zero times, when counted with orientation.
Thus, by link connected sums along pairs of positive and
negative strands passing through such a neighbourhood, followed by
an isotopy, we may assume that each component of $L_{\eta}$ is contained
entirely in one of the Seifert fibred components~$S_{ij}$.
The collection of all link components contained in a single $S_{ij}$
defines a class in $H_1(S_{ij})$.

In summary, the homology class $\PD(d^2(\xi_0,\eta))\in H_1(M)$
can be written as a sum of homology classes coming from the
inclusions $S_{ij}\subset M$. As above, for this inclusion to make
sense we have to remove an open ball from the $S_{ij}$ where
the connected sum with $M_{i\pm1}$ is performed.

Now the $H_1$-completeness of the $S_{ij}$ (Proposition~\ref{prop:Seifert-H1})
allows us to modify $\xi_0$ into a Bott-integrable contact structure
on $M$ homotopic to $\eta$ as a tangent $2$-plane field as
in the proof of Lemma~\ref{lem:H1-ot}. This completes
the proof of Theorem~\ref{thm:main} for the case under consideration.

Observe that, by Yano's criterion, we may assume the link $L_{\eta}$ to
be a graph link. It will then remain a graph link after the
described modification, since (self) link connected sum preserves
the graph link property. However, the sublink of critical elliptic
orbits that $H_1$-completeness allows us to construct need not be
equal to this graph link; in general,
it merely represents the same class in homology.
\subsection{Graph manifolds containing $(S^1\times S^2)$-summands}
Finally, we come to the general case of the `if' direction
in Theorem~\ref{thm:main}. Write $M=N\#k(S^1\times S^2)$ with
$N$ a graph manifold prime to $S^1\times S^2$.

Let $\eta$ be a tangent $2$-plane field on $M$ with
\[ \PD(e(\eta))=(a_{\eta},b_{\eta})\in H_1(N)\oplus\Z^k=H_1(M)\]
represented by a graph link. On $N$, as before, we choose a Bott-integrable
contact structure that is $H_1$-complete on the Seifert
fibred pieces coming from the JSJ decompositions of the
irreducible summands of~$N$. On $k(S^1\times S^2)$ we may choose,
by \cite[Theorem~1.9]{ghs24}, a Bott-integrable contact structure
whose Euler class has the Poincar\'e dual~$b_{\eta}$. Write
$\xi_0$ for the Bott-integrable contact structure on $M$
obtained via the connected sum.

Then $\PD(e(\xi_0))=(a_0,b_{\eta})$ for some $a_0\in H_1(N)$.
Now, the class
\[ (a_0-a_{\eta},0)=\PD\bigl(e(\xi_0)-e(\eta)\bigr)\]
is representable by a graph link, and it equals
$2\PD(d^2(\xi_0,\eta))$. This allows us to write
\[ \PD\bigl(d^2(\xi_0,\eta)\bigr)=(a_{0\eta},0)\]
with $2a_{0\eta}=a_0-a_{\eta}$.
Moreover, by Theorem~\ref{thm:Yano2},
the class $\rho^N_*(a_0-a_{\eta})\in H_1(\calC_N)$ is divisible by $m(0)=0$,
which forces it to equal~$0$. As above, using the fact that
$H_1(\calC_N)$ is a free abelian group, we conclude that $\rho^N_*(a_{0\eta})=0$.
As in the previous case we then argue that the Poincar\'e dual
$(a_{0\eta},0)$ of the obstruction class $d^2(\xi_0,\eta)$
can be represented by a link with each component contained in
a Seifert fibred piece of~$N$, and the proof concludes as before.
\section{The mapping torus of Arnold's cat map}
\label{section:cat}
In this section we illustrate many of the notions used in the
proof of Theorem~\ref{thm:main} with an instructive example.

The diffeomorphism $\psi$ of the $2$-torus $T^2=\R^2/\Z^2$ defined by
the matrix
\[ A=\begin{pmatrix}
2 & 1\\
1 & 1
\end{pmatrix}
\in\SL(2,\Z),\]
known as Arnold's cat map, possesses several interesting dynamical features,
in spite of its apparent simplicity. For instance, $\psi$ is an
Anosov diffeomorphism, an ergodic map, and mixing. The torus
automorphism $\psi$ is of hyperbolic type, with eigenvalues
$\lambda:=(3+\sqrt{5})/2$ and $(3-\sqrt{5})/2=\lambda^{-1}$.

Let $M=M(\psi)$ be the mapping torus of $\psi$, that is, the torus bundle
over $S^1$ defined as the quotient space of $T^2\times\R$ under the
$\Z$-action generated by
\[ (x,y;t)\longmapsto (2x+y,x+y;t-1).\]
The $T^2$-fibration of $M$ is induced by the map $(x,y;t)\mapsto t$.
This aspherical (and hence irreducible) $3$-manifold carries Sol geometry
and is not Seifert fibred~\cite{scot83}.
\subsection{Topological data of $M(\psi)$}
The inclusion $T^2\rightarrow M$ of a fibre is $\pi_1$-injective
by the homotopy exact sequence of the fibration $M\rightarrow S^1$,
so the fibre $T^2\subset M$ is incompressible. In particular,
$M$ is not atoroidal.
This implies that the JSJ decomposition requires cutting along
at least one torus, and indeed one will suffice: when an open
tubular neighbourhood of a fibre is removed, the complement
$T^2\times[\varepsilon,1-\varepsilon]$ is $S^1$-fibred.
So the JSJ complex of $M$ consists of a single $1$-cell
attached to a single $0$-cell.

The cat map $\psi$ has a single fixed point $(x,y)=(0,0)$.
So the standard cell decomposition of $T^2$ (with the zero
cell taken as the fixed point of~$\psi$) gives rise to a cell
decomposition of $M$ with an additional $(k+1)$-cell for each
$k$-cell of~$T^2$. This allows one to read off the following
presentation of the fundamental group:
\[ \pi_1(M)\cong\langle a,b,c\,|\,[a,b],\, acb^{-1}a^{-2}c^{-1},\,
bcb^{-1}a^{-1}c^{-1}\rangle.\]
The first homology group $H_1(M)$, as the abelianisation of
$\pi_1(M)$, is then isomorphic to~$\Z$, generated by the
loop in $t\mapsto (0,0;t)$, $t\in [0,1]$, corresponding to~$c$.

It follows that Yano's homomorphism $\rho_*\co H_1(M)\rightarrow
H_1(\calC_M)$ is an isomorphism. By Theorem~\ref{thm:Yano1},
every graph link in $M$ must be homologically trivial.
As a first consequence we note that, thanks to Lemmas \ref{lem:Euler-formula}
and \ref{lem:Lf-graph}, any Bott-integrable contact structure
on $M(\psi)$ must have trivial Euler class.
\subsection{A family of integrable contact structures on~$M(\psi)$}
Consider a contact form
\[ \alpha=h_1(t)\,\rmd x+h_2(t)\,\rmd y\]
of Lutz type on $T^2\times\R$. The condition for this to be invariant
under the $\Z$-action with quotient $M(\psi)$ is that
\begin{equation}
\label{eqn:h1h2}
\left\{\begin{array}{rcl}
h_1(t) & = & 2h_1(t-1)+h_2(t-1),\\
h_2(t) & = & h_1(t-1)+h_2(t-1).
\end{array}\right.
\end{equation}
On the other hand, if one takes $u,v$ to be coordinates on $\R^2$
with respect to a basis of eigenvectors of $A$ for the
eigenvalues $\lambda,\lambda^{-1}$, the invariant $1$-form $\alpha$
should look like
\[ \alpha=\lambda^t\,\rmd u+\lambda^{-t}\,\rmd v,\]
which is more obviously adapted to Sol geometry~\cite{scot83}.

With some linear algebra, one arrives at the following
proposition.

\begin{prop}
A countable family $\alpha_n$, $n\in\N_0$, of Bott-integrable
contact forms on $M(\psi)$ can be defined by the formula
\begin{eqnarray*}
\alpha_n & = &
\frac{\sqrt{2}}{\sqrt{5}}\Biggl[
\Bigl(\sin\bigl(\frac{\pi}{4}+2n\pi t\bigr)\lambda^t-
\cos\bigl(\frac{\pi}{4}+2n\pi t\bigr)\lambda^{-t}\Bigr)\,\rmd x\\
&& \mbox{}+\Bigl(\frac{\sqrt{5}-1}{2}\sin\bigl(\frac{\pi}{4}+2n\pi t\bigr)
\lambda^t+ \frac{\sqrt{5}+1}{2}\cos\bigl(\frac{\pi}{4}+2n\pi t\bigr)
\lambda^{-t}\Bigr)\,\rmd y\Biggr].
\end{eqnarray*}
\end{prop}

\begin{proof}
The verification of the invariance condition \eqref{eqn:h1h2}
is straightforward. The factor $\sqrt{2}/\sqrt{5}$ has been chosen
for purely aesthetic reasons; it ensures that $h_1(0)=0$ and
$h_2(0)=1$, so that the sequence
\[ h_2(1),h_1(1),h_2(2),h_1(2),h_2(3),h_1(3),\ldots\]
equals the Fibonacci sequence $1,1,2,3,5,8,\ldots$

One computes that
\[ h_1(t)h_2'(t)-h_1'(t)h_2(t)=-\frac{2}{\sqrt{5}}\bigl(
2n\pi+\cos(4n\pi t)\log\lambda\bigr),\]
which is negative for all $n\in\N_0$ and $t\in\R$, so the
$\alpha_n$ are contact forms.

A Bott integral for $\alpha_n$ is given by the lift of any
Morse function in the variable $t\in\R/\Z$.
\end{proof}

As tangent $2$-plane fields, the contact structures $\ker\alpha_n$
are all homotopic to the trivial $2$-plane bundle $\ker\rmd t$. The Reeb
vector fields of the contact forms $\alpha_n$ are tangent to the fibres.
\subsection{The classification of Bott-integrable contact structures
on~$M(\psi)$}
The special linear group $\SL(2,\Z)$ is generated by the matrices
\[ S=\begin{pmatrix}
0  & 1 \\
-1 & 0
\end{pmatrix}
\;\;\;\text{and}\;\;\;
T=\begin{pmatrix}
1 & 1 \\
0 & 1
\end{pmatrix}.\]
The matrix $A$ defining the Arnold cat map can be factorised as
\[ A=-ST^{-1}ST^{-2}S,\]
which is conjugate to $T^{-3}S$. Therefore, by Honda's
classification (up to isotopy) of tight contact structures
on torus bundles, see specifically the table on page~90 of~\cite{hond00},
all tight contact structures on $M(\psi)$ are universally tight, that is,
their lift to~$\R^3$ is tight.
There is a unique tight structure for each value $2n\pi$, $n\in\N$,
of its Giroux torsion \cite[Definition~1.2]{giro00},
and two with zero Giroux torsion. For the relation between the
Giroux torsion and the `twisting in the $S^1$-direction' used by
Honda, see \cite[page~86]{hond00}.

Giroux's classification~\cite{giro00} of tight contact structures
on torus bundles is up to diffeomorphism, and according to
\cite[Theorem~1.3]{giro00}, there is a unique one for each value
$2n\pi$ of the Giroux torsion, including the case $n=0$.

The discrepancy between the classification up to isotopy and that up to
diffeomorphism comes from the diffeomorphism of $M(\psi)$ which acts
as minus the identity on the torus fibres and changes the coorientation
of the contact structure. For positive torsion, this can be effected by an
isotopy (a shift transverse to the fibres), but not for zero torsion.

Given this information, we can now formulate the complete
classification of the Bott-integrable contact structures on~$M(\psi)$.

\begin{prop}
The $\ker\alpha_n$, $n\in\N_0$, constitute a complete list of
the tight contact structures on $M(\psi)$ and, as we have seen, they are
all Bott integrable. An overtwisted contact structure on $M(\psi)$
is Bott integrable if and only if its Euler class is trivial. There
is a countable family of such overtwisted structures, distinguished by
their $3$-dimensional homotopy invariant.
\end{prop}

\begin{proof}
On $\R^3$, the $1$-form $\alpha_n$, just like any $1$-form of Lutz type,
defines the standard tight contact structure; see \cite[Lemme~1]{giro99}
for the simple argument. For $n=0$, the coefficient function $h_2$
of $\rmd y$ is always positive, so the Giroux torsion of $\ker\alpha_0$
is zero. An increase of $n$ by $1$ introduces one additional
full twist around the origin by the curve $t\mapsto (h_1(t),h_2(t))$
on the interval $t\in [0,1]$, which implies an additional $2\pi$ Giroux
torsion.

As there is no $2$-torsion in $H^2(M(\psi))$, the overtwisted contact
structures are classified by their Euler class and their $3$-dimensional
homotopy invariant. As we have seen, the Bott-integrable ones must
have trivial Euler class, so we can obtain all of them simply
by taking the connected sum of $\bigl(M(\psi),\ker\alpha_0\bigr)$, say,
with any of the overtwisted contact structures on~$S^3$, which are
Bott integrable by \cite[Theorem~1.9]{ghs24}.
\end{proof}

Notice that triviality of the Euler class means that the contact structures
are trivial as abstract $2$-plane bundles, but as tangent $2$-plane
fields they are distinguished by the $3$-dimensional homotopy invariant.
\begin{ack}
This work is part of a project in the
Son\-der\-for\-schungs\-be\-reich TRR 191
\textit{Symplectic Structures in Geometry, Algebra and Dynamics},
funded by the DFG (Projektnummer 281071066 -- TRR 191). J.~H.\ is
partially supported by the Fondation Courtois.
\end{ack}


\begin{thebibliography}{10}
%
\bibitem{alve16}
\textsc{M. R. R. Alves},
Cylindrical contact homology and topological entropy,
\textit{Geom. Topol.}
\textbf{20} (2016), 3519--3569.
%
\bibitem{ach19}
\textsc{M. R. R. Alves, V. Colin and K. Honda},
Topological entropy for Reeb vector fields in dimension three via open
book decompositions,
\textit{J. \'Ec. polytech. Math.}
\textbf{6} (2019), 119--148.
%
\bibitem{bofo04}
\textsc{A. V. Bolsinov and A. T. Fomenko},
\textit{Integrable Hamiltonian Systems -- Geometry, Topology, Classification},
Chapman \& Hall/CRC, Boca Raton, FL (2004).
%
\bibitem{elia89}
\textsc{Ya. Eliashberg},
Classification of overtwisted contact structures on $3$-manifolds,
\textit{Invent. Math.}
\textbf{98} (1989), 623--637.
%
\bibitem{geig08}
\textsc{H. Geiges},
\textit{An Introduction to Contact Topology},
Cambridge Stud. Adv. Math. \textbf{109},
Cambridge University Press, Cambridge (2008).
%
\bibitem{ghs24}
\textsc{H. Geiges, J. Hedicke and M. Sa\u{g}lam},
Bott-integrable Reeb flows on $3$-manifolds,
\textit{J. Lond. Math. Soc. (2)}\/
\textbf{109} (2024), Paper No. e12859, 42 pp.
%
\bibitem{geth23}
\textsc{H. Geiges and N. Thies},
Klein bottles in lens spaces,
\textit{Involve}
\textbf{16} (2023), 621--636.
%
\bibitem{giro99}
\textsc{E. Giroux},
Une infinit\'e de structures de contact tendues sur une infinit\'e
de vari\'et\'es,
\textit{Invent. Math.}
\textbf{135} (1999), 789--802.
%
\bibitem{giro00}
\textsc{E. Giroux},
Structures de contact en dimension trois
et bifurcations des feuilletages de surfaces,
\textit{Invent. Math.}
\textbf{141} (2000), 615--689.
%
\bibitem{hatc23}
\textsc{A. Hatcher},
\textit{Notes on Basic $3$-Manifold Topology},
Cornell University, last revised in 2023;
available at \verb+https://pi.math.cornell.edu/~hatcher/+
%
\bibitem{hond00}
\textsc{K. Honda},
On the classification of tight contact structures~II,
\textit{J. Differential Geom.}
\textbf{55} (2000), 83--143.
%
\bibitem{jash79}
\textsc{W. H. Jaco and P. B. Shalen},
Seifert fibred spaces in $3$-manifolds,
\textit{Mem. Amer. Math. Soc.}
\textbf{21}, no.~220 (1979).
%
\bibitem{jane83}
\textsc{M. Jankins and W. D. Neumann},
\textit{Lectures on Seifert Manifolds},
Brandeis Lecture Notes \textbf{2},
Brandeis University, Waltham, MA (1983);
available at\\
\verb+http://www.math.columbia.edu/~neumann/preprints/+
%
\bibitem{joha79}
\textsc{K. Johannson},
\textit{Homotopy Equivalences of $3$-Manifolds with Boundaries},
Lecture Notes in Math. \textbf{761},
Springer-Verlag, Berlin (1979).
%
\bibitem{kato09}
\textsc{A. Katok},
Open problems in elliptic dynamics,
transcript of a talk at the Penn State seminar
\textit{Dynamics and its working tools},
notes by V.~Climenhaga,
September~1 (2009).
%
\bibitem{masc11}
\textsc{L. Macarini and F. Schlenk},
Positive topological entropy of Reeb flows on spherizations,
\textit{Math. Proc. Cambridge Philos. Soc.}
\textbf{151} (2011), 103--128.
%
\bibitem{mepa78}
\textsc{W. H. Meeks, III, and J. Patrusky},
Representing homology classes by embedded circles on a compact surface,
\textit{Illinois J. Math.}
\textbf{22} (1978), 262--269.
%
\bibitem{pate91}
\textsc{G. P. Paternain},
Entropy and completely integrable Hamiltonian systems,
\textit{Proc. Amer. Math. Soc.}
\textbf{113} (1991), 871--873.
%
\bibitem{scot83}
\textsc{P. Scott},
The geometries of $3$-manifolds,
\textit{Bull. London Math. Soc.}
\textbf{15} (1983), 401--487.
%
\bibitem{wald67I}
\textsc{F. Waldhausen},
Eine Klasse von $3$-dimensionalen Mannigfaltigkeiten~I,
\textit{Invent. Math.}
\textbf{3} (1967), 308--333.
%
\bibitem{wald67II}
\textsc{F. Waldhausen},
Eine Klasse von $3$-dimensionalen Mannigfaltigkeiten~II,
\textit{Invent. Math.}
\textbf{4} (1967), 87--117.
%
\bibitem{yano85a}
\textsc{K. Yano},
The support of global graph links,
\textit{J. Math. Soc. Japan}
\textbf{37} (1985), 683--702.
%
\bibitem{yano85b}
\textsc{K. Yano},
Homology classes which are represented by graph links,
\textit{Proc. Amer. Math. Soc.}
\textbf{93} (1985), 741--746.
%
\bibitem{yau06}
\textsc{M.-L. Yau},
Vanishing of the contact homology of overtwisted contact $3$-manifolds,
with an appendix by Ya. Eliashberg,
\textit{Bull. Inst. Math. Acad. Sin. (N.S.)}
\textbf{1} (2006), 211--229.
%
\end{thebibliography}
\end{document}